\documentclass[11pt]{amsart}
\usepackage{amsmath}
\usepackage{amsthm}
\usepackage{amsfonts}
\usepackage{amssymb}
\usepackage{hyperref}

\newtheorem{thm}{Theorem}[section]
\newtheorem{lem}[thm]{Lemma}

\newtheorem{cor}[thm]{Corollary}
\theoremstyle{definition}

\numberwithin{equation}{section}

\let\Lm=\Lambda
\let\abs=\envert
\newcommand{\floor}[1]{\left\lfloor#1\right\rfloor}

\theoremstyle{remark}
  
\newtheorem{rem}[thm]{Remark}

\begin{document}
\title{On finiteness of odd superperfect numbers}
\author{Tomohiro Yamada}
\address{Center for Japanese language and culture, Osaka University,
562-8558, 8-1-1, Aomatanihigashi, Minoo, Osaka, Japan}
\email{tyamada1093@gmail.com}

\date{}
\maketitle

\begin{abstract}
Some new results concerning the equation $\sigma(N)=aM, \sigma(M)=bN$ are proved.
As a corollary, there are only finitely many odd superperfect numbers
with a fixed number of distinct prime factors.
\end{abstract}

\section{Introduction}\label{intro}

As usual, $\sigma(N)$ shall denote the sum of positive divisors of a positive integer $N$
and call a positive integer $N$ to be perfect if $\sigma(N)=2N$. Though
it is not known whether or not an odd perfect number
exists, many conditions which must be satisfied by
such a number are known.

Analogous to this notion, Suryanarayana \cite{Sur1} called
$N$ to be superperfect if $\sigma(\sigma(N))=2N$.
Suryanarayana showed that if $N$ is even superperfect, then
$N=2^m$ with $2^{m+1}-1$ prime, and if $N$ is odd superperfect,
then $N$ must be square and have at least two distinct prime factors.

Dandapat, Hunsucker and Pomerance \cite{DHP} showed that
if $\sigma(\sigma(N))=kN$ for some integer $k$ and $\sigma(N)$ is a prime power,
then $N$ is even superperfect or $N=21, k=3$.  Later
Pomerance \cite{Pom1} called $N$ to be super multiply perfect
if $\sigma(\sigma(N))=kN$ for some integer $k$ and showed that
if $p^m\mid\sigma(N)$ and $N\mid\sigma(p^m)$ for some prime power $p^m$,
then $N=2^{n-1}$ or $2^n-1$ with $2^n-1$ prime or $N=15, 21, 1023$.

In the West Coast Number Theory Conference 2005, the author posed the question
whether there exist only finitely many odd integers $N$ such that $N\mid\sigma(\sigma(N))$
and $\omega(\sigma(N))=s$ for each fixed $s$ \cite{Ymd1}, where $\omega(n)$ denotes
the number of distinct prime factors of $n$.  The above-mentioned result of
Dandapat, Hunsucker and Pomerance answers the special case $s=1$ of this question affirmatively.

Concerning the unitary divisor sum $\sigma^*(N)$
($d$ is called a unitary divisor of $N$ if $d\mid N$ and $d, N/d$ are coprime),
the author already proved that $N=9, 165$ are all the odd integers
satisfying $\sigma^*(\sigma^*(N))=2N$ \cite{Ymd2}.

In this paper, although we cannot prove the above-mentioned conjectures, some results are proved.
Before stating our results, we introduce the notation $C_i(\ldots)$ for $i=0, 1, 2, \ldots$,
each of which denotes some effectively computable positive constant depending only on its arguments.

\begin{thm}\label{th11}
If a quadruple of integers $N, M, a, b$ satisfies $\sigma(N)=aM$, $\sigma(M)=bN$ and
$\omega(\sigma(N))\leq k$, then we have $a, b<C_0(k)$ for some effectively computable constant $C_0(k)$
depending only on $k$.
\end{thm}

\begin{thm}\label{th12}
Additionally to the condition described in Theorem \ref{th11}, assume that both $M$ and $N$ are odd.
Then each of $M$ and $N$ must have a prime factor smaller than some effectively
computable constant $C_1(k)$ depending only on $k$.
\end{thm}

\begin{rem}
The additional condition would be necessary.
Indeed, if we allow $M$ or $N$ to be even and take $N=2^m, M=2^{m+1}-1$,
then Theorem 1.2 would imply that there exist only finitely many Mersenne primes,
contary to the widely believed conjecture!
\end{rem}

\begin{thm}\label{th13}
For any given integers $a, b, k, k^\prime$, there are only finitely many pairs of odd integers $M, N$
satisfying $\sigma(N)=aM, \sigma(M)=bN$ with $\omega_1(N)\leq k^\prime$ and $\omega(M)\leq k$,
where $\omega_1(N)$ denotes the number of primes dividing $N$ only once.
Moreover, such integers $M, N$ can be bounded by some effectively computable constant $C_2(a, b, k, k^\prime)$
depending only on $a, b, k, k^\prime$.
\end{thm}

Using Suryanarayana's result that an odd superperfect number must be square, the latter result
gives the following corollary, which implies our conjecture in the case $\sigma(\sigma(N))=2N$.
Moreover, we observe that if $N$ is an odd superperfect number,
then $\sigma(N)$ must be a product of a square and a prime power $p^e$
with $p\equiv e\equiv 1\pmod{4}$ and $\sigma(\sigma(\sigma(N)))=\sigma(2N)=3\sigma(N)$.
Hence, we have the following finiteness result.

\begin{cor}\label{cor14}
For each fixed $k$, there exists some effectively computable constant $B_k$ such that,
if $N$ is an odd superperfect number with either $\omega(N)\leq k$ or $\omega(\sigma(N))\leq k$,
then $N\leq B_k$.
\end{cor}

The corresponding result for odd perfect numbers has been known for years.
Dickson \cite{Dic} proved that there exist only finitely many odd perfect numbers $N$ with $\omega(N)\leq k$
for each fixed $k$ and an effective upper bound was given by Pomerance \cite{Pom2},
improved by Nielsen \cite{Nie1, Nie2}.

It will be relatively easy to show that there exist only finitely many odd superperfect numbers $N$
with \textit{both} $\omega(N)$ and $\omega(\sigma(N))$ fixed, using Dickson's argument.
However, if we only require that $\omega(N)$ \textit{or} $\omega(\sigma(N))$ is fixed,
we have only known that $\omega(N)\geq 2$ from \cite{Sur2} and $\omega(\sigma(N))\geq 2$ from \cite{DHP},
which is implied by the result of \cite{Pom1} mentioned above.
Our corollary states that, even under this restriction, we can obtain an effective upper bound
for an odd superperfect number $N$.

Our argument in this paper is based upon the one in \cite{Ymd2}.
In \cite{Ymd2}, we used the fact that if $\sigma^*(\sigma^*(N))=2N$,
then $N$ must be factored into $N=\prod_i p_i^{e_i}$
with $p_i^{e_i}+1=2^{a_i}q^{b_i}$ for some integers $a_i, b_i, q$.
This means that $p_i^{e_i}$'s must be distributed very thin and therefore
the product of $\sigma^*(p_i^{e_i})/p_i^{e_i}$'s must be small.

However, we deal with the $\sigma$ function in this paper.
For a small prime $p$, $\sigma(p^e)/p^e$ must be fairly large
and therefore our argument from \cite{Ymd2} does not work.

We introduce some preliminary notations.
In order to prove Theorems \ref{th11}-\ref{th13}, we consider slightly more general situation.
Assume that $N$ is an integer such that $\omega(\sigma(N))=k$
and we let $q_1<q_2< \cdots< q_k$ to be the prime divisors of $\sigma(N)$.
For each $1\leq r\leq k$ and a prime power $l$, let $S_{r, l}=S_{r, l}(q_1, q_2, \ldots, q_r)$
denote the set of prime divisors $p$ of $N$ such that
$p^e\mid\mid N$ with $l\mid (e+1)$ and
\begin{equation}\label{eq00}
\frac{p^l-1}{p-1}=\sigma(p^{l-1})=\prod_{i=1}^r q_i^{a_i}
\end{equation}
for some integers $a_i (1\leq i\leq r)$ with $a_r\neq 0$ and let $S_r=\cup_l S_{r, l}$,
where $l$ runs over all prime powers.
Clearly, each prime divisor of $N$ must belong to a set $S_{r, l}$
for some $1\leq r\leq k$ and a prime power $l$.

In Section \ref{main}, using a lower bound for linear forms of logarithms, we shall show that
each $S_r$ contains at most $r$ small primes.
Combined with Lemma \ref{lmdi}, which states that the contribution of large prime factors
to the size of $\sigma(M)\sigma(N)/MN$ must be very small, we shall prove the following fact.
\begin{thm}\label{th0}
Let $N$ be an integer such that $\omega(\sigma(N))=k$
and let $q_1<q_2<\cdots <q_k$ be the prime divisors of $\sigma(N)=aM$ as introduced above.
Then, for every $r=1, 2, \ldots, k$, $N$ has at most $r$ prime factors in $S_r$
below $C_4(r, q_r)=\exp (C_3(r)(\log q_r/\log\log q_r)^{1/2(r+1)})$ and
\begin{equation}\label{eq0}
\sum_{\substack{p\in S_r,\\ p\geq C_4(r, q_r)}}\frac{1}{p}< \exp \left(-C_5(r)\left(\frac{\log q_r}{\log\log q_r}\right)^\frac{1}{2(r+1)}\right).
\end{equation}
\end{thm}

This theorem allows us to overcome the above-mentioned obstacle.
Indeed, it is not difficult to derive Theorem \ref{th11} from Theorem \ref{th0},
as shown in Section \ref{Pr11}.
With the aid of a diophantine inequality shown in Section \ref{appr}, we shall prove
Theorems \ref{th12} and \ref{th13}.

\section{Preliminary lemmas}\label{lemmas}

In this section, we introduce some preliminary lemmas.

The first lemma is a special case of Matveev\cite[Theorem 2.2]{Mat},
which gives a lower bound for linear forms of logarithms.
We use this lemma to prove our gap principle in Section \ref{main}.
The second lemma describes an elementary property of values of cyclotomic polynomials.

\begin{lem}\label{lmll}
Let $a_1, a_2, \ldots, a_n$ be positive integers with $a_1>1$.
For each $j=1, \ldots, n$, let $A_j\geq\max\{0.16, \log a_j\}$.
Let $b_1, b_2, \ldots, b_n$ be arbitrary integers.

Put
\begin{equation}
\begin{split}
B&=\max \{1, \abs{b_1}A_1/A_n, \abs{b_2}A_2/A_n, \ldots, \abs{b_n} \},\\
\Omega&=A_1A_2\ldots A_n,\\
C(n)&=\frac{16}{n!}e^n(2n+3)(n+2)(4(n+1))^{n+1}\left(\frac{1}{2}en\right)(4.4n+5.5\log n+7) \\
\end{split}
\end{equation}
and
\begin{equation}
\Lm=b_1\log a_1+\ldots+b_n\log a_n.
\end{equation}
Then we have $\Lm=0$ or
\begin{equation}
\log\abs{\Lm}>-C(n)(1+\log 3-\log 2+\log B)\max\left\{1, \frac{n}{6}\right\}\Omega.
\end{equation}
\end{lem}

\begin{rem}
The assumption that $a_1>1$ is added in order to ensure that
$\log a_1, \log a_2, \ldots, \log a_r$ are linearly independent over the integers
for some $r (1\leq r\leq n)$.
We note that we do not need recent results for linear forms in logarithms.
We can see that a lower bound of the form $\log\abs{\Lm}>-B^{1/g(n)}\log^{f(n)} A$,
where $f(n), g(n)$ are effectively computable functions of $n$ such that $g(n)>1$,
is strong enough for our purpose.  Such an estimate would increase the right hand side
of \eqref{eqfd0} but still give an estimate that $\log n_r<\log^{h(r)} m_{r+1}$
for some effectively computable function $h(r)$.
Thus even an old estimate such as Fel'dman\cite{Fel} suffices.
\end{rem}

\begin{lem}\label{lm22}
If $a, l$ are positive integers with $a\geq 2, l\geq 3$ and $(a, l)\neq (2, 6)$,
then $(a^l-1)/(a-1)$ must have at least $\tau(l)-1$ distinct prime factors,
where $\tau(l)$ denotes the number of divisors of $l$.
Moreover, at least one of such prime factors is congruent to $1\pmod{l}$.
\end{lem}

\begin{proof}
A well known result of Zsigmondy \cite{Zsi} states that if $a\geq 2, n\geq 3$ and $(a, n)\neq (2, 6)$,
then $a^n-1$ has a prime factor which does not divide $a^m-1$ for any $m<n$.
Applying this result to each divisor $d>1$ of $l$, we see that
$(a^l-1)/(a-1)$ must have at least $\tau(l)-1$ prime factors.
In particular, applying with $n=l$, we obtain a prime factor $q$ such that
$a\pmod{q}$ has order $l$ and therefore $q\equiv 1\pmod{l}$.
\end{proof}

\section{The distribution of large primes in $S_{r, l}$}\label{largeprimes}
In this section, we shall give an upper bound for the sum $\sum_{p\in S_{r, l}, p>X} 1/p$ for each fixed $r, l$.

\begin{lem}\label{lmdi0}
Let $p_0, p_1, p_2$ be distinct primes with $p_0$ odd, $l$ and $f$ be positive integers
and put $H_i=\floor{f\log p_0/\log p_i}$ for $i=1, 2$.
If the congruence
\begin{equation}\label{eqdi1}
p_i^l\equiv 1\pmod{p_0^f}
\end{equation}
holds for $i=1, 2$, then
\begin{equation}
\frac{1}{2}H_1H_2\leq \gcd (l, p_0^{f-1}(p_0-1)).
\end{equation}
\end{lem}

\begin{proof}
It is clear that $p_1^{a_1}p_2^{a_2}$ takes distinct values modulo $p_0^f$
for all nonnegative integers $a_1$ and $a_2$ with $0\leq a_1\log p_1+a_2\log p_2<f\log p_0$.
So that $p_1^{a_1}p_2^{a_2}$ takes at least $H_1H_2/2$ distinct values modulo $p_0^f$.
But these can take at most $\gcd (l, p_0^{f-1}(p_0-1))	$ distinct values since
both $p_1$ and $p_2$ have residual orders dividing $\gcd (l, p_0^{f-1}(p_0-1))$ modulo $p_0^f$ by \eqref{eqdi1}.
Hence, we obtain $H_1H_2/2\leq \gcd (l, p_0^{f-1}(p_0-1))$.
\end{proof}

\begin{lem}\label{lmdi1}
Let $p_0, p_1, p_2$ be distinct primes with $p_2>p_1$ and $p_0>2$ and $q, s$ be positive integers.
If $l$ is an integer greater than $5s^2$ and there are integers $f_1, f_2$ such that $p_0^{f_i}\mid\sigma(p_i^{l-1})$ and $p_0^{f_i}\geq\sigma(p_i^{l-1})^{1/s}$ for $i=1, 2$, then
\begin{equation}\label{eqdi5}
\log p_2>\frac{5}{4}\log p_1.
\end{equation}
\end{lem}

\begin{proof}
Assume that $p_1, p_2$ are two distinct primes satisfying the assumption in the lemma
but $\log p_2\leq (5/4)\log p_1$, contrary to the statement of the lemma.
Let $f=\min \{f_1, f_2\}$ and $H_i=\floor{f\log p_0/\log p_i}$.
We observe that $p_0^{f_i}\geq\sigma(p_i^{l-1})^{1/s}>p_i^{(l-1)/s}$ for $i=1, 2$
and therefore $f\log p_0>(l-1)\log p_1/s$.
Since we have assumed that $l>5s^2$, we have
\begin{equation}
H_1\geq \floor{\frac{l-1}{s}}\geq 5s
\end{equation}
and
\begin{equation}
H_2\geq \floor{\frac{(l-1)\log p_1}{s\log p_2}}\geq \floor{\frac{4(l-1)}{5s}}\geq 4s.
\end{equation}
By the definition of $H_i$, we can easily see that
\begin{equation}
H_1>\frac{5sf\log p_0}{(5s+1)\log p_1}\geq \frac{5f\log p_0}{6\log p_1}
\end{equation}
and
\begin{equation}
H_2>\frac{4sf\log p_0}{(4s+1)\log p_2}\geq \frac{4f\log p_0}{5\log p_2}.
\end{equation}
Hence, Lemma \ref{lmdi0} gives
\begin{equation}
l\log p_1\log p_2>\frac{1}{3}f^2 \log^2 p_0>\frac{(l-1)^2}{3s^2}\log^2 p_1.
\end{equation}
Recalling the assumption that $l>5s^2$, we have
\begin{equation}
\frac{\log p_2}{\log p_1}>\frac{(l-1)^2}{3ls^2}\geq \frac{25s^2}{3(5s^2+1)}\geq\frac{25}{18}>\frac{5}{4},
\end{equation}
which contradicts to the assumption.  Hence, we have $p_2>p_1^{5/4}$.
\end{proof}

Using this result, we obtain the following inequality.

\begin{lem}\label{lmdi}
For any set $S_{r, l}$ defined in the introduction and $X>2$, we have
\begin{equation}
\sum_{p>X, p\in S_{r, l}}\frac{1}{p}<\frac{C_6(r)\log^r X}{X}.
\end{equation}
\end{lem}

\begin{rem}
It is well known that, for fixed integers $l>2, r$ and fixed primes $q_1, q_2, \ldots, q_r$,
there exist only finitely many primes $p$ and integers $a_1, a_2, \ldots, a_r$
satisfying \eqref{eq00}.
Combining Coates' theorem \cite{Coa} and Schinzel's theorem \cite{Scz},
it follows that such integers and, consequently, the elements of $S_{r, l}$
are bounded by an effectively computable constant depending on $l$ and the $q_i$'s.
For details of the history of the largest prime factor of polynomial values,
see Chapter 7 of Shorey and Tijdeman's book \cite{ST}.
Furthermore, two theorems of Evertse \cite{Ev1}\cite{Ev2} imply that
$\abs{S_{r, l}}$ is bounded by an effectively computable constant depending on $r, l$.
However, in this paper, we need a result depending only on $r$.
\end{rem}

\begin{proof}
First we note that $S_{r, l}$ can be divided into $r$ sets $S_{r, l, j} (1\leq j\leq r)$ so that
if $p\in S_{r, l, j}$, then $q_j^{f_j}\mid\sigma(p^{l-1})$ for an integer $f_j$
such that $q_j^{f_j}\geq\sigma(p^{l-1})^{1/r}$.

Assume that $l>5r^2$.
If $p_1<p_2$ are two primes belonging to $S_{r, l, j}$, then
$\log p_2>(5/4)\log p_1$ by Lemma \ref{lmdi1}.
Hence, we obtain
\begin{equation}
\sum_{\substack{p>X,\\ p\in S_{r, l, j}}}\frac{1}{p}<\sum_{i=0}^\infty \frac{1}{X^{(5/4)^i}}<\frac{4}{X}
\end{equation}
and therefore
\begin{equation}
\sum_{\substack{p>X,\\ p\in S_{r, l}}}\frac{1}{p}<\sum_{j=1}^r\sum_{\substack{p>X,\\ p\in S_{r, l, j}}}\frac{1}{p}<\frac{2r}{X}<\frac{C_6(r)\log^r X}{X}.
\end{equation}

Next assume that $l\leq 5r^2$.
It is clear that the number of primes $p<x$ belonging to $S_{r, l}$
is at most $(l\log x)^r/\prod_{i=1}^r \log q_i$ and partial summation gives
\begin{equation}
\sum_{p>X, p\in S_{r, l}}\frac{1}{p}<\frac{1}{X}+\int_X^\infty\frac{(l\log t)^r dt}{t^2\prod_{i=1}^r \log q_i}< C_7(r)\frac{(l\log X)^r}{X}.
\end{equation}
Since $l\leq 5r^2$, we have $(l\log X)^r\leq (5r^2\log X)^r$ and therefore
\begin{equation}
\sum_{p>X, p\in S_{r, l}}\frac{1}{p}<C_6(r)\frac{\log^r X}{X}.
\end{equation}
\end{proof}

\section{Main theory - proof of Theorem {\ref{th0}}}\label{main}

In this section, we shall prove Theorem \ref{th0}, which plays the most essential role in this paper.
We begin by proving the following lemma.

\begin{lem}\label{lmfe}
Let $r, l_1, \ldots, l_{r+1}$ and $n_1<n_2< \ldots <n_r$ be positive integers.
Let $m_1<m_2< \cdots <m_{r+1}$ be distinct primes.
Assume that there exist integers $a_{ij} (1\leq i\leq s+1, 1\leq j\leq s)$ such that
\begin{equation}\label{eqfd1}
\frac{m_j^{l_j}-1}{m_j-1}=\prod_{i=1}^{r} n_i^{a_{ij}}
\end{equation}
for $j=1, 2, \ldots, r+1$ and $a_{ij}>0$ for the index $i$ for which $n_i$ assumes the maximum.
Then we have
\begin{equation}\label{eqfd0}
\log n_r<C_8(r)(\log^{2(r+1)} m_{r+1})(\log\log m_{r+1}).
\end{equation}
\end{lem}

\begin{proof}
We put
\begin{equation}
\Lm_j=-l_j\log m_j+\log (m_j-1)+\sum_i a_{ij}\log n_i=\log\left(\frac{m_j^{l_j}-1}{m_j^{l_j}}\right)\neq 0
\end{equation}
for each $j=1, 2, \ldots, r+1$.  Since $\Lm_j\neq 0$, using Matveev's lower bound given in Lemma \ref{lmll}
we obtain
\begin{equation}
\log \abs{\Lm_j}>-C(r+2)\log\left(\frac{3el_j\log m_j}{2\log n_r}\right)\log^2 m_j\prod_{i=1}^r \log n_i.
\end{equation}
Observing that $\abs{\Lm_j}<1/(m_j^{l_j}-1)$, we have
\begin{equation}\label{eqfd11}
l_j<C_9(r)\log m_j \log^r n_r\log(\log^2 m_j\log^{r-1} n_r)
\end{equation}
and therefore
\begin{equation}\label{eqfd12}
a_{ij}<l_j\frac{\log m_j}{\log n_i}<C_{10}(r)\log^2 m_j \log^r n_r\log(\log m_j \log n_r).
\end{equation}

Putting $A=C_{10}(r)\log^2 m_{r+1} \log^r n_r\log(\log m_{r+1}\log n_r)$,
we see that \eqref{eqfd1} ensures the existence of integers $g_1, \ldots, g_{r+1}$
not all zero with absolute values at most $((r+1)^{1/2}A)^r$ such that
\begin{equation}
\prod_{j=1}^{r+1}\left(\frac{m_j^{l_j}-1}{m_j-1}\right)^{g_j}=1
\end{equation}
by an improved form of Siegel's lemma (the original form of Siegel's lemma gives the upper bound $1+((r+1)A)^r$.
For detail, see Chapter I of \cite{Sch}).

We put
\begin{equation}
\Lm=\sum_{j=1}^{r+1} g_j l_j \log m_j-g_j\log (m_j-1).
\end{equation}
Since
\begin{equation}
\Lm=\sum_{j=1}^{r+1} g_j\log\left(\frac{m_j^{l_j}}{m_j-1}\right)=\sum_{j=1}^{r+1}g_j\log\left(\frac{m_j^{l_j}}{m_j^{l_j}-1}\right)
\end{equation}
and $m_j^{l_j}-1$ must be divisible by $n_r$ by assumption, we have
\begin{equation}\label{eqfd13}
\begin{split}
\abs{\Lm}< & \sum_{j=1}^{r+1}\frac{g_j}{m_j^{l_j}-1}\leq\frac{2(r+1)(rA)^r}{n_r} \\
< & 2(r+1)^2(rC_{10}(r))^r \\
& \times \frac{\log^{2r} m_{r+1} \log^{r^2} n_r \log^r(\log m_{r+1}\log n_r)}{n_r}.
\end{split}
\end{equation}

We observe that $\Lm$ does not vanish since $e^\Lm=\prod_{j=1}^{r+1}m_j^{l_jg_j}/(m_j-1)^{g_j}$
must be divisible by the largest prime $m_t$ among $m_j$'s for which $l_tg_t\neq 0$.
Hence, taking $G=\max\{\abs{g_j l_j \log m_j/\log m_{r+1}}\mid 1\leq j\leq r+1 \}$,
for which we have
\begin{equation}
G<C_{11}(r)(r+1)^{r/2}\log^{2r+3} m_{r+1} \log^{r^2+1} n_r
\end{equation}
from \eqref{eqfd12}, we can apply Matveev's theorem to $\Lm$ and obtain
\begin{equation}\label{eqfd14}
\begin{split}
\log\abs{\Lm}\geq & -C(2(r+1))\left(\log \left(\frac{3}{2}eG\right)\right)\prod_{j=1}^{r+1}(\log m_j)^2 \\
\geq & -C_{12}(r)\log(\log m_{r+1}\log n_r)\log^{2(r+1)} m_{r+1}.
\end{split}
\end{equation}
Now, combining inequalities \eqref{eqfd13} and \eqref{eqfd14}, we have
\begin{equation}
\log n_r<C_8(r)(\log^{2(r+1)} m_{r+1})(\log\log m_{r+1}),
\end{equation}
which proves the lemma.
\end{proof}

We see that the former part of Theorem \ref{th0} is an immediate consequence of this lemma.
Indeed, taking $p_1<p_2<\cdots<p_{r+1}$ to be any $r+1$ prime factors of $N$
and applying Lemma \ref{lmfe} with $m_i=p_i$ for $i=1, 2, \ldots , r+1$ and
$n_j=q_j$ for $j=1, 2, \ldots, r$, we must have
$\log p_{r+1}>C_3(r)(\log q_r/\log\log q_r)^{1/2(r+1)}$.
Thus it remains to prove \eqref{eq0}.
By Lemma \ref{lmdi}, we have, for each $r, l$,
\begin{equation}
\sum_{\substack{p\in S_{r, l},\\ p\geq C_4(r, q_r)}}\frac{1}{p}<\exp\left(-C_{13}(r)\left(\frac{\log q_r}{\log\log q_r}\right)^\frac{1}{2(r+1)}\right).
\end{equation}
Since $l$ must be a prime power dividing one of $(q_i-1)$'s $(1\leq i\leq r)$ by Lemma \ref{lm22},
there exist at most $\log^r q_r$ choices for $l$.
Hence, we obtain
\begin{equation}\label{eqs1}
\begin{split}
\sum_{\substack{p\in S_r,\\ p\geq C_4(r, q_r)}}\frac{1}{p}< & (\log q_r)^r \exp\left(-C_{13}(r)\left(\frac{\log q_r}{\log\log q_r}\right)^\frac{1}{2(r+1)}\right) \\
< & \exp \left(-C_5(r)\left(\frac{\log q_r}{\log\log q_r}\right)^\frac{1}{2(r+1)}\right).
\end{split}
\end{equation}
This completes the proof of Theorem \ref{th0}.

\section{Proof of Theorem {\ref{th11}}}\label{Pr11}

We may assume that $\sigma(N)$ has exactly $k$ distinct prime factors.
By Theorem \ref{th0}, there exist at most $k(k+1)/2$ prime factors $p$ of $N$
for which $p\in S_r$ and $p<C_4(r, q_r)$ for some $r$.
Let $T$ be the set of such primes.
Then, summing \eqref{eq0} over $r=1, 2, \ldots, k$, we obtain
\begin{equation}
\sum_{p\mid N, p\not\in T}\frac{1}{p}=\sum_{r=1}^k \sum_{\substack{p\in S_r,\\ p\geq C_4(r, q_r)}}\frac{1}{p}<\exp \left(-C_{14}(k)\left(\frac{\log q_1}{\log\log q_1}\right)^\frac{1}{2(k+1)}\right).
\end{equation}

Since the sum of reciprocals of the first $k(k+1)/2$ primes is \\ $<C_{15}\log\log k$,
we have
\begin{equation}
\begin{split}
\sum_{p\mid N}\frac{1}{p}= & \sum_{p\in T}\frac{1}{p}+\sum_{p\mid N, p\not\in T}\frac{1}{p} \\
< & C_{15}\log\log k+\exp \left(-C_{14}(k)\left(\frac{\log q_1}{\log\log q_1}\right)^\frac{1}{2(k+1)}\right).
\end{split}
\end{equation}
Hence, $\sum_{p\mid N}(1/p)<C_{16}(k)$.
Clearly we have $\sum_{p\mid M}(1/p)<C_{17}(k)$ since $M$ has at most $k$ distinct prime factors.
Now Theorem \ref{th11} immediately follows from the observation
that $\sigma(N)/N<\prod_{p\mid N} p/(p-1)<\exp (\sum_{p\mid N} 1/(p-1))<\exp (\sum_{p\mid N} (2/p))$.

\section{Approximation of rational numbers}\label{appr}
In this section, we shall prove a lemma concerning diophantine approximation which is used to prove Theorem \ref{th12} and \ref{th13}.
We shall begin with introducing some notations.
For each prime $p$, we let $h(p^g)=\sigma(p^g)/p^g$ for $g=1, 2, \ldots$ and $h(p^\infty)=p/(p-1)$.
Moreover, for not necessarily distinct primes $p_1, p_2, \ldots, p_k$
and $e_1, e_2, \ldots, e_k\in \{0, 1, 2, \ldots, \infty\}$,
we let $h(p_1^{e_1}, p_2^{e_2}, \ldots, p_k^{e_k})=\prod_{i=1}^k h(p_i^{e_i})$.

We observe that if $p_1, p_2, \ldots, p_k$ are distinct primes and $e_1, e_2, \ldots, e_k$
are nonnegative integers, then
\begin{equation}
h(p_1^{e_1}, p_2^{e_2}, \ldots, p_k^{e_k})=\frac{\sigma(p_1^{e_1}p_2^{e_2}\cdots p_k^{e_k})}{p_1^{e_1}p_2^{e_2}\cdots p_k^{e_k}}.
\end{equation}
For brevity, we write $h(p_1^{x_1}, p_2^{x_2}, \ldots, p_k^{x_k})=h_k(p^x)$
and $h(p_1^{\infty}, p_2^{\infty}, \ldots, p_k^{\infty})=h_k(p^{\infty})$.

For a rational number $\alpha$ and (not necessarily distinct) primes $p_1, \ldots, p_k$,
let $s_k(\alpha, p)=s(\alpha, p_1, \ldots, p_k)$ be the infimum of numbers
of the form $\alpha-h_k(p^e)$ with $e_1, \ldots, e_k$ such that
$h_l(p^e)=\alpha$ for some (not necessarily distinct) primes
$p_{k+1}, p_{k+2}, \ldots, p_l$ and exponents $e_{k+1}, e_{k+2}, \ldots, e_l$.
Moreover, let $s(\alpha; k)$ be the infimum of $s_k(\alpha, p)$
with $p_1, \ldots, p_k$ running over all primes.

We shall prove that $s(\alpha; k)$ can be bounded from below by an effectively
computable positive constant depending only on $\alpha$ and $k$.
This result is essentially included in Theorem 4.2 of \cite{Pom2}.
But we reproduce the proof of this lemma since our lemma allows duplication of primes and,
as Pomerance notes in p. 204 in \cite{Pom2}, the proof can be much shortened
when $p_1, p_2, \ldots, p_k$ are all odd.

\begin{lem}\label{lma1}
For any rational number $\alpha=n/d>1$ and primes $p_1, p_2, \ldots, p_k$,
we have $s_k(n/d, p)>\delta_k(n, p)$,
where $\delta_k(n, p)=\delta(n, p_1, p_2, \ldots, p_k)$ is an effectively computable
positive constant depending only on $p_1, p_2, \ldots, p_k$ and $n$.
\end{lem}

\begin{proof}
For $k=0$, we clearly have $s(\alpha)=\alpha-1\geq 1/d>1/n$.
Now we shall give a lower bound for $s_k(\alpha, p)$ in terms of
$\alpha, p_1, p_2, \ldots, p_k$ and $s_{k-1}(\alpha, p)$.
This inductively prove the lemma.

We first see that $h_k(p^\infty)\neq\alpha$.
Indeed, the denominator of $h_k(p^\infty)$ is even
while the denominator of $\alpha=h_k(p^e)$ must be odd.
So that it suffices to deal two cases $h_k(p^\infty)<\alpha$
and $h_k(p^\infty)>\alpha$.

In the former case, we see that $\alpha-h_k(p^e)
>\alpha-h_k(p^\infty)\geq 1/(d\prod_i(p_i-1))>1/(n\prod_i(p_i-1))$.
Thus we have $s_k(n/d, p)\geq 1/(n\prod_i(p_i-1))$
in this case.

In the latter case, letting $x_j=\floor{\log(2kn\prod_i(p_i-1))/\log p_j}$,
we see that $e_j<x_j$ for some $j$ since
\begin{equation}
\begin{split}
h_k(p^x)= & \prod_i h_k(p^\infty)\prod_i\left(1-\frac{1}{p_i^{x_i+1}}\right) \\
\geq & \prod_i h_k(p^\infty)\left(1-\sum_i\frac{1}{p_i^{x_i+1}}\right) \\
> & \left(\alpha+\frac{1}{d\prod_i(p_i-1)}\right)\left(1-\frac{1}{2n\prod_i(p_i-1)}\right) \\
= & \alpha\left(1+\frac{1}{n\prod_i(p_i-1)}\right)\left(1-\frac{1}{2n\prod_i(p_i-1)}\right) \\
> & \alpha. \\
\end{split}
\end{equation}

In this case, we have $s_k(\alpha, p)\geq \inf\{s(\alpha/h(p_i^{e_i}), p_1, p_2, \ldots, \hat{p_i}, p_k)|1\leq i\leq k, 1\leq e_j<x_j (1\leq j\leq k, j\neq i) \}$.
Observing that the reduced numerator of $\alpha/h(p_i^{e_i})$ divides $np_i^{e_i}$,
we can take $\delta_k(n, p)=\min\{1/(n\prod_i(p_i-1)), \inf\{\delta(np_i^{e_i}, p_1, p_2, \ldots, \hat{p_i}, p_k)|1\leq i\leq k, 1\leq e_j<x_j (1\leq j\leq k)\}\}$.
By induction, this completes the proof.
\end{proof}

\begin{lem}\label{lmap}
Let $n, d$ be integers with $d$ odd, $p_1, \ldots, p_s$ any odd (not necessarily distinct) primes
and $e_1, \ldots, e_s$ non-negative integers.
Assume that $h(p_1^{e_1}, p_2^{e_2}, \ldots, p_s^{e_s})<n/d$ but
$h(p_1^{e_1}, p_2^{e_2}, \ldots, p_l^{e_l})=n/d$ for some (not necessarily distinct) primes
$p_{s+1}, p_{s+2}, \ldots, p_l$ and positive integral exponents $e_{s+1}, e_{s+2}, \ldots, e_l$.
Then the inequality
\begin{equation}
\frac{n}{d}-\prod_{i=1}^{s} h(p_i^{e_i})>C_{18}(s, n)
\end{equation}
holds for effectively computable constants $C_{18}(s, n)$ depending only on $s, n$.
\end{lem}

\begin{proof}
For $s=0$, we have a trivial estimate $C_{18}(0, n)\geq n/(n-1)-1>1/n$.

Next we shall show that we can compute $C_{18}(s+1, n)$ in term of $C_{18}(s, n)$.
This gives the lemma by induction.  If
\begin{equation}\label{eq61}
p_i>\frac{2n}{dC_{18}(s, n)}-1
\end{equation}
for some $i$, then we have
\begin{equation}
h(p_i^{e_i})<\left(1-\frac{1}{p_i}\right)^{-1}<\left(1-\frac{\frac{C_{18}(s, n)}{2}}{\frac{n}{d}-\frac{C_{18}(s, n)}{2}}\right)^{-1}=\frac{\frac{n}{d}-\frac{C_{18}(s, n)}{2}}{\frac{n}{d}-C_{18}(s, n)}
\end{equation}
and therefore the inductive hypothesis yields that
\begin{equation}
\prod_{i=1}^{s+1}h(p_i^{e_i})\leq h(p_i^{e_i})\left(\frac{n}{d}-C_{18}(s, n)\right)<\frac{n}{d}-\frac{C_{18}(s, n)}{2}.
\end{equation}
If \eqref{eq61} does not hold for any $i$, then we have 
\begin{equation}
\prod_{i=1}^{s+1}h(p_i^{e_i})<\frac{n}{d}-\min\delta(n, p_1, \ldots, p_{s+1}),
\end{equation}
where $p_1, \ldots, p_{s+1}$ run all primes below $2n/(dC_{18}(s, n))$.
Hence, we have
\begin{equation}
\frac{n}{d}-\prod_{i=1}^{s+1}h(p_i^{e_i})>\min \left\{\frac{C_{18}(s, n)}{2}, \bar\delta(s, n)\right\},
\end{equation}
where $\bar\delta(s, n)$ denotes the minimum value of $\delta(n, p_1, \ldots, p_{s+1})$
with $p_i\leq 2n/C_{18}(s, n)$.
Now, Lemma \ref{lma1} ensures that $C_{18}(s+1, n)$ is positive and effectively computable.
\end{proof}

\section{Proof of Theorem {\ref{th12}}}

First we shall show that $M$ must have a prime factor smaller than $C_1(k)$.
Let $T$ be the same set as defined in Section \ref{Pr11}.
By Theorem \ref{th0}, $T$ contains at most $k(k+1)/2$ primes.
Since $N$ is odd, we can apply Lemma \ref{lmap} taking $p_1, p_2, \ldots, p_s$ to be the primes in $T$
and $p_{s+1}, \ldots, p_l$ to be the prime factors of $MN$ not in $T$, with primes dividing both $M$ and $N$
counted doubly, to obtain $\prod_{p\in T}h(p)<ab-C_{19}(k, ab)$.
Hence, $\prod_{p\mid MN, p\not\in T}h(p)>1+C_{20}(k, ab)$, implying
\begin{equation}
\sum_{p\mid MN, p\not\in T}\frac{1}{p}>C_{21}(k, ab).
\end{equation}
But, as in the proof of Theorem \ref{th11}, we have
\begin{equation}
\sum_{p\mid N, p\not\in T}\frac{1}{p}<\exp \left(-C_{14}(k)\left(\frac{\log q_1}{\log\log q_1}\right)^\frac{1}{2(k+1)}\right).
\end{equation}
Hence, observing that $\omega(M)\leq \omega(\sigma(N))\leq k$, we have $q_1<C_{22}(k, ab)$.
Since $a, b<C_0(k)$, we have $q_1<C_1(k)$.

Next we shall show that $N$ must have a prime factor smaller than $C_1(k)$.
$X$ shall denote the smallest prime factor of $N$.
Let $Q$ be an arbitrary real number which shall be chosen later
and $s$ be the index satisfying $q_s\leq Q<q_{s+1}$.
Similarly to \eqref{eqs1}, there exist at most $r$ primes below $C_4(r, q_r)$ in $S_r$ and
we have
\begin{equation}
\begin{split}
\sum_{r=s+1}^k\sum_{\substack{p\in S_r,\\ p\geq C_4(r, q_r)}}\frac{1}{p}< & \exp \left(-C_5(r)\left(\frac{\log q_r}{\log\log q_r}\right)^\frac{1}{2(r+1)}\right) \\
< & \exp \left(-C_5(r)\left(\frac{\log Q}{\log\log Q}\right)^\frac{1}{2(r+1)}\right).
\end{split}
\end{equation}
Hence, for any real $X$, we obtain
\begin{equation}
\sum_{r=s+1}^k\sum_{p\in S_r, p\geq X}\frac{1}{p}<\frac{k(k+1)}{2X}+\exp \left(-C_5(k)\left(\frac{\log Q}{\log\log Q}\right)^\frac{1}{2(k+1)}\right).
\end{equation}

Since $q_s\leq Q$, Lemma \ref{lmdi} gives that
\begin{equation}
\sum_{p\in S_r, p\geq X}\frac{1}{p}<\frac{(\log Q\log X)^r}{X}
\end{equation}
for each $r\leq s$.

Hence, we have
\begin{equation}
\begin{split}
\sum_{p\mid N}\frac{1}{p}< & \frac{s(\log Q\log X)^s}{X}+\frac{k(k+1)}{2X} \\
& +\exp \left(-C_5(k)\left(\frac{\log Q}{\log\log Q}\right)^\frac{1}{2(k+1)}\right).
\end{split}
\end{equation}
Taking $Q$ so that $C_5(k)(\log Q/\log\log Q)^{1/2(k+1)}=\log X$, we have
\begin{equation}
\sum_{p\mid N}\frac{1}{p}<\frac{C_{23}(k)\log^{2k^2+k} X}{X}.
\end{equation}

However, since $\sigma(M)\sigma(N)/MN=ab$ and $M$ is odd with $\omega(M)\leq k$,
Lemma \ref{lmap} gives that $\sigma(M)/M<ab-C_{18}(k, ab)$ and therefore $\sigma(N)/N>1+C(k, ab)$, implying that
\begin{equation}
\sum_{p\mid N}\frac{1}{p}>C_{24}(k, ab).
\end{equation}
Hence, we must have $X<C_{25}(k, ab)$.  Since $a, b<C_0(k)$ by Theorem \ref{th11},
we have $X<C_1(k)$, which proves Theorem \ref{th12}.

\section{Proof of Theorem {\ref{th13}}}

First we shall show that $q_s<C_{26}(a, b, k, k^\prime, s)$ by induction.
The inductive base is that $q_1<C_1(k)$, which is the former part of Theorem \ref{th12}.
Now, it suffices to prove that for any positive integer $s\leq k-1$ we have $q_{s+1}<C_{26}(a, b, k, k^\prime, s+1)$
under the assumption that $q_1, \ldots, q_s<C_{26}(a, b, k, k^\prime, s)$.

We see that for each $r\leq s$, $S_r$ contains at most a bounded number of primes.
Each $S_{r, l}$ with $l\geq 3$ contains at most $C_{27}(r, l)$ primes by two theorems of
Evertse\cite{Ev1}\cite{Ev2}.
Since $q_1, q_2, \ldots, q_s\leq C_{26}(a, b, k, k^\prime, s)$ by the inductive hypothesis,
we see that $l\leq q_s$ is also bounded by $C_{28}(a, b, k, k^\prime, s)$.
By assumption, for each $r\leq s$, except at most $k^\prime$ primes, any prime in $S_r$ belongs to
some $S_{r, l}$ for some prime power $l\geq 3$.
Hence, each $S_r$ with $r\leq s$ contains at most $C_{29}(a, b, k, k^\prime, r)$ primes.
We note that, by virtue of the inductive assumption that $q_1, q_2, \ldots, q_s\leq C_{26}(a, b, k, k^\prime, s)$,
we can also use classical finiteness results such as Bugeaud and Gy\H{o}ry\cite{BG},
Coates\cite{Coa} and Kotov\cite{Kot}.

Moreover, by Theorem \ref{th0}, for $r>s$, $S_r$ contains at most $r$ prime factors below
$C_4(r, q_r)$.

Now, let $U_s$ be the set of prime factors $p$ dividing $N$ at least twice
for which $p\geq C_4(r, q_r)$ and $p\in S_r$ for some $r>s$.
It follows from the above observations that there exist at most $C_{30}(a, b, k, k^\prime, s)+s(s+1)/2+k+k^\prime$ primes
outside $U_s$ dividing $MN$ and therefore Lemma \ref{lmap} yields that
\begin{equation}\label{eq81}
\sum_{p\in U_s}\frac{1}{p}=\sum_{r=s+1}^k \sum_{\substack{p\in S_r,\\ p\geq C_4(r, q_r)}}\frac{1}{p}>C_{31}(a, b, k, k^\prime, s).
\end{equation}

However, \eqref{eq0} gives that
\begin{equation}
\sum_{\substack{p\in S_r,\\ p\geq C_4(r, q_r)}}\frac{1}{p}< \exp \left(-C_5(r)\left(\frac{\log q_r}{\log\log q_r}\right)^\frac{1}{2(r+1)}\right)
\end{equation}
and therefore
\begin{equation}\label{eq82}
\begin{split}
\sum_{p\in U_s}\frac{1}{p}= & \sum_{r=s+1}^k \sum_{\substack{p\in S_r,\\ p\geq C_4(r, q_r)}}\frac{1}{p} \\
< & \exp \left(-C_{32}(k)\left(\frac{\log q_{s+1}}{\log\log q_{s+1}}\right)^\frac{1}{2(k+1)}\right).
\end{split}
\end{equation}

In order that both \eqref{eq81} and \eqref{eq82} simultaneously hold, we must have $q_{s+1}<C_{26}(a, b, k, k^\prime, s+1)$,
which completes our inductive argument to prove that $q_j<C_{26}(a, b, k, k^\prime, j)$ for every $j=1, 2, \ldots, k$.

Now, by virtue of Lemma \ref{lm22},
\begin{equation}
\frac{p^l-1}{p-1}=q_1^{e_1}q_2^{e_2}\cdots q_k^{e_k}
\end{equation}
implies that $l<q_k<C_{26}(a, b, k, k^\prime, j)$ and, using
classical finiteness results such as Bugeaud and Gy\H{o}ry\cite{BG}, Coates\cite{Coa} and Kotov\cite{Kot},
we finally obtain $p<C_{33}(l, q_1, q_2, \ldots, q_k)<C_{34}(a, b, k, k^\prime)$.
This proves the theorem.

\section{Concluding remarks}
Our proof of Theorem \ref{th13} exhibited in the last section indicates that
we can explicitly give the upper bound for $N$ in terms of $a, b, k, k^\prime$;
although Evertse's results\cite{Ev1}\cite{Ev2} are not effective for the {\em size} of solutions,
these results give an effective upper bounds for the {\em number} of solutions.
However, the upper bound which our proof yields would become considerably large
due to its inductive nature exhibited in the last section.  For sufficiently large $k$,
our proof yields that
\begin{equation}
N<\exp\exp\ldots\exp (a+b+k+k^\prime),
\end{equation}
where the number of iterations of the exponential function is $\ll k$ and $\gg k$.

{}

\begin{thebibliography}{}
\bibitem{BG}
Y. Bugeaud and K. Gy\H{o}ry, Bounds for the solutions of Thue-Mahler equations and norm form equations, Acta Arith. \textbf{74} (1996), 273--292.
\bibitem{Coa}
J. Coates, An effective $p$-adic analogue of a theorem of Thue II; The greatest prime factor of a binary form, Acta Arith. \textbf{16} (1969/70), 399--412.
\bibitem{DHP}
G. G. Dandapat, J. L. Hunsucker and Carl Pomerance, Some new results on odd perfect numbers, Pacific J. Math. \textbf{57} (1975), 359--364.
\bibitem{Dic}
L. E. Dickson, Finiteness of the odd perfect and primitive abundant numbers with $n$ distinct prime factors, Amer. J. Math. \textbf{35} (1913), 413--422.
\bibitem{Ev1}
J. -H. Evertse, On equations in $S$-units and the Thue-Mahler equation, Invent. Math. \textbf{75} (1984), 561--584.
\bibitem{Ev2}
J. -H. Evertse, The number of solutions of the Thue-Mahler equation, J. Reine Angew. Math. \textbf{482} (1997), 121--149.
\bibitem{Fel}
N. I. Fel'dman, Improved estimate for a linear form of the logarithms of algebraic numbers, Mat. Sb. \textbf{77} (1968), 423--436 = Math. USSR Sb. \textbf{6 }(1968), 393--406.
\bibitem{Kot}
S. V. Kotov, Greatest prime factor of a polynomial, Math. Zametki \textbf{13} (1973), 515--522. =Math. Notes \textbf{13} (1973), 313--317.
\bibitem{Mat}
E. M. Matveev, An explicit lower bound for a homogeneous rational linear form in the logarithms of algebraic numbers. II, Izv. Ross. Akad. Nauk Ser. Mat. 64 (2000), 125--180, Eng. trans., Izv. Math. \textbf{64} (2000), 127--169.
\bibitem{Nie1}
P. P. Nielsen, An upper bound for odd perfect numbers, Integers \textbf{3} (2003), \#A14.
\bibitem{Nie2}
P. P. Nielsen, Odd perfect numbers, diophantine equation, and upper bounds, Math. Comp. \textbf{84} (2015), 2549--2567.
\bibitem{Pom1}
C. Pomerance, On multiply perfect numbers with a special property, Pacific J. Math. \textbf{57} (1975), 511--517.
\bibitem{Pom2}
C. Pomerance, Multiply perfect numbers, Mersenne primes and effective computability, Math. Ann. \textbf{226} (1977), 195--206.
\bibitem{Sch}
A. Schmidt, Diophantine approximations and diophantine equations, Lecture Notes in Math. \textbf{1467}, Springer-Verlag, Berlin, 1996.
\bibitem{Scz}
A. Schinzel, On two theorems of Gelfond and some of their applications, Acta Arith. \textbf{13} (1967), 177--236.
\bibitem{ST}
T. N. Shorey and R. Tijdeman, Exponential diophantine equations, Cambridge University Press, Cambridge, 1986.
\bibitem{Sur1}
D. Suryanarayana, Super perfect numbers, Elem. Math. \textbf{24} (1969), 16--17.
\bibitem{Sur2}
D. Suryanarayana, There is no odd super perfect number of the form $p^{2\alpha}$, Elem. Math. \textbf{28} (1973), 148--150.
\bibitem{Ymd1}
T. Yamada, Problem 005:10, Western Number Theory Problems, 2005, \url{https://wcnt.files.wordpress.com/2018/02/wcnt-problems-2005.pdf}.
\bibitem{Ymd2}
T. Yamada, Unitary super perfect numbers, Math. Pannon. \textbf{19} (2008), 37--47.
\bibitem{Zsi}
K. Zsigmondy, Zur Theorie der Potenzreste, Monatsh. f{\"u}r Math. \textbf{3} (1882), 265--284.
\end{thebibliography}
\end{document}